\documentclass[11pt]{amsart}
\usepackage{amsmath,amssymb,amsthm,upref,graphicx,mathrsfs, enumerate}
\usepackage{esint}
\usepackage{textcomp} 
\usepackage{array} 
\usepackage{textcomp} 
\usepackage[
  colorlinks=true,
  linkcolor=blue,
  citecolor=blue,
  urlcolor=blue,
  pagebackref]{hyperref}
\usepackage{enumitem}

\usepackage{amsthm}
\let\oldproofname=\proofname
\renewcommand{\proofname}{\rm\bf{\oldproofname}}

  \makeatletter
\def\specialsection{\@startsection{section}{1}%
  \z@{\linespacing\@plus\linespacing}{.5\linespacing}%
  {\normalfont}}
\def\section{\@startsection{section}{1}%
  \z@{.7\linespacing\@plus\linespacing}{.5\linespacing}%
  {\normalfont\scshape}}
\makeatother

\numberwithin{equation}{section}
\allowdisplaybreaks

\newtheorem{theorem}{Theorem}[section]

\newtheorem{lemma}[theorem]{Lemma}
\newtheorem{cor}[theorem]{Corollary}

\theoremstyle{definition}

\newtheorem{remark}[theorem]{Remark}
\theoremstyle{theorem}

\def\XXint#1#2#3{{\setbox0=\hbox{$#1{#2#3}{\int}$}
\vcenter{\hbox{$#2#3$}}\kern-.5\wd0}}

\begin{document}

\title[On embeddings between spaces of functions of generalized bounded variation]{On embeddings between spaces of functions of generalized bounded variation}

\author[G. H. Esslamzadeh and M. Moazami Goodarzi]{G. H. Esslamzadeh and M. Moazami Goodarzi}

\address[G. H. Esslamzadeh]{Department  of Mathematics, Faculty of Sciences, Shiraz University, Shiraz 71454, Iran}
\email{esslamz@shirazu.ac.ir}

\address[M. Moazami Goodarzi]{Department of Mathematics, Faculty of Sciences, Shiraz University, Shiraz 71454, Iran}
\email{milad.moazami@gmail.com}


\date{\today}


\begin{abstract}
In this note, we aim to establish a number of embeddings between various function spaces that are frequently considered in the theory of Fourier series. More specifically,
we give sufficient conditions for the embeddings $\Phi V[h]\subseteq \Lambda\text{BV}^{(p_n\uparrow p)}$, $\Lambda V[h_1]^{(p)}\subseteq\Gamma V[h_2]^{(q)}$ and
$\Lambda\text{BV}^{(p_n\uparrow p)}\subseteq\Gamma\text{BV}^{(q_n\uparrow q)}$. Our results are new even for the well-known spaces that have been studied
in the literature.
In particular, a number of results due to M. Avdispahi\'{c}, that describe relationships between the classes $\Lambda\text{BV}$ and $V[h]$, are derived as special cases.

\end{abstract}

\maketitle


\section{\bf{Introduction and preliminaries}}\label{Sect:1}

The Jordan class BV of functions of bounded variation has been generalized by many authors in various ways (see \cite{Appell}). In particular, Schembari and Schramm
introduced the space $\Phi V[h]$ in \cite{SS} to encompass previous generalizations:

Let $\Phi=\{\phi_j\}_{j=1}^\infty$ be a sequence of increasing convex functions on $[0,\infty)$ such that $\phi_j(0)=0$ for all $j$, and $0<\phi_{j+1}(x)\leq\phi_j(x)$
for all $x>0$. If $h$ is a nondecreasing sequence of positive reals, we say that $\Phi$ is a Schramm sequence (with respect to $h$) provided that for each $x>0$,
$\sum_{j=1}^\infty\frac{\phi_j(x)}{h(n)}=\infty$ as $n\rightarrow\infty$. A real-valued function $f$ on $[a,b]$ is said to be of bounded $\Phi$-$h$-variation if
$$
V_{\Phi,h}(f)=\sup_{1\leq n<\infty}\frac{v(n,\Phi,f)}{h(n)}<\infty,
$$
where $v(n,\Phi,f)=v(n,\Phi,f,[a,b])$ is the $\Phi$-modulus of variation of $f$, that is, the supremum of the sums $\sum_{j=1}^n \phi_j(|f(I_j)|)$, taken over all finite collections $\{I_j\}_{j=1}^n$ of
nonoverlapping subintervals of $[a,b]$ and $f(I_j)=f(\sup I_j)-f(\inf I_j)$. We denote by $\Phi V[h]$ the linear space of all functions $f$ on $[a,b]$ such
that $V_{\Phi,h}(cf)<\infty$ for some constant $c>0$.

It is shown in \cite{SS} that $\Phi V[h]$ is indeed a Banach space with respect to the norm
$$
\|f\|_{\Phi,h}:=|f(a)|+\inf\{c>0:V_{\Phi,h}(f-f(a)/c)\leq1\}.
$$
This space has many applications in Fourier analysis
as well as in treating topics such as integration, convergence, summability, etc. (see e.g. \cite{Wt1,SW2,SS}).

If $\phi$ is a strictly increasing convex function on $[0,\infty)$ with $\phi(0)=0$, and if $\Lambda=\{\lambda_j\}_{j=1}^\infty$ is
a Waterman sequence (i.e., $\Lambda$ is a nondecreasing sequence of positive numbers such that $\sum_{j=1}^\infty\frac1{\lambda_j}=\infty\:$), by taking $\phi_j(x)=\phi(x)/\lambda_j$ for all $j$, we get the class $\phi\Lambda\text{BV}$ of functions of $\phi\Lambda$-bounded variation. This
class was introduced by Schramm and Waterman in \cite{SW} (see also \cite{L}). More specifically, if $\phi(x)=x^p$ ($p\geq1$), we get the Waterman-Shiba class
$\Lambda\text{BV}^{(p)}$, which was introduced by Shiba in \cite{S}. When $p=1$, we obtain the well-known Waterman class $\Lambda\text{BV}$. Also, if $h$ is
a modulus of variation (i.e., a nondecreasing and concave sequence of positive reals) and $\phi_j(x)=x$ for all $j$, the Chanturiya class $V[h]$ is obtained as
a special subclass of $\Phi V[h]$.

In the literature, much attention has been devoted to the study of relationships between the above-mentioned classes; see \cite{Wt1}, \cite{Perlman}, \cite{Avdis},
\cite{KY}, \cite{GHM} and the references therein for some results in this direction. In particular, a characterization of embeddings between
$\Lambda\text{BV}$ classes was obtained by Perlman and Waterman \cite{Perlman}. Ge and Wang characterized the embeddings $\Lambda\text{BV}\subseteq\phi\text{BV}$ and $\phi\text{BV}\subseteq\Lambda\text{BV}$.
Kita and Yoneda showed in \cite{KY} that the embedding
$\text{BV}_p\subseteq\text{BV}^{(p_n\uparrow\infty)}$ is both automatic and strict for all $1\leq p<\infty.$
Furthermore, Goginava characterized the embedding $\Lambda\text{BV}\subseteq\text{BV}^{(q_n\uparrow\infty)}$, and a characterization of
the embedding $\Lambda\text{BV}^{(p)}\subseteq\text{BV}^{(q_n\uparrow q)}$ ($1\leq q\leq\infty$) was given by Hormozi, Prus-Wi\'{s}niowski and Rosengren in \cite{HPR}.
More recently, the embeddings $\Lambda\text{BV}^{(p)}\subseteq \Gamma\text{BV}^{(q_n\uparrow q)}$ and $\Phi\text{BV}\subseteq\text{BV}^{(q_n\uparrow q)}$ ($1\leq q\leq\infty$) were investigated
by Goodarzi, Hormozi and Memi\'{c} (see \cite{GHM}).


\section{\bf{Results}}\label{Sect:2}

Our first main result presents a sufficient condition for the embedding $\Phi V[h]\subseteq \Lambda\emph{BV}^{(p_n\uparrow p)}$ (see Theorem (\ref{t1}) below).
Before that, we need a lemma.

If $\Phi=\{\phi_j\}_{j=1}^\infty$ is a Schramm sequence, we define $\Phi_k(x):=\sum_{j=1}^k\phi_j(x)$ for $x\geq0$. Then $\Phi_k(x)$ is clearly an increasing
convex function on $[0,\infty)$ such that $\Phi_k(0)=0$ and $\Phi_k(x)>0$ for $x>0$. Without loss of generality we assume that $\Phi_k(x)$ is strictly
increasing on $[0,\infty)$. We denote by $\Phi_k^{-1}(x)$ the inverse function of $\Phi_k(x)$. If $\lambda=\{\lambda_j\}$ and $\Gamma=\{\gamma_j\}$ are
Waterman sequences, for each $n$ we define $\Lambda(n):=\sum_{j=1}^n\frac1{\lambda_j}$ and $\Gamma(n):=\sum_{j=1}^n\frac1{\gamma_j}$.

\begin{lemma}\label{lemma1}
Let $1<q<\infty$ and $k\in\mathbb{N}$. If $f\in \Phi V[h]$ and $x_1,x_2,...,x_k$ are nonnegative real numbers such that
$$
\sum_{j=1}^k \phi_j(x_{\tau(j)})\leq v(k,\Phi,f)
$$
for any permutation $\tau$ of $k$ letters, then
\begin{equation}\label{conc}
\Big(\sum_{j=1}^k x_j^q\Big)^{\frac1{q}}\leq 16\left(1+V_{\Phi,h}(f)\right)\max_{1\leq m\leq k} m^{\frac1{q}}\Phi_m^{-1}\big(h(k)\big).
\end{equation}
\end{lemma}
\begin{proof}
Note first that following the arguments in the proof of \cite[Theorem 2.1]{Wu} one can verify that
$$
\Big(\sum_{j=1}^k x_j^q\Big)^{\frac1{q}}\leq 16\max_{1\leq m\leq k} m^{\frac1{q}}\Phi_m^{-1}\big(v(k,\Phi,f)\big).
$$
On the other hand, since the $\Phi_m^{-1}$ are strictly increasing concave functions with $\Phi_m^{-1}(0)=0$, we get
$$
\Phi_m^{-1}(at)\leq(1+a)\Phi_m^{-1}(t), \ \ \ \text{for any} \ \ \ a,t>0.
$$
Now, applying the latter inequality with
$$
a:=V_{\Phi,h}(f) \ \ \ \text{and} \ \ \ t:=h(k)
$$
yields (\ref{conc}), as desired.
\end{proof}

\begin{theorem}\label{t1}
The embedding $\Phi V[h]\subseteq \Lambda\emph{BV}^{(p)}$ holds whenever
$$
\sum_{k=1}^\infty \Delta\Big(\frac1{\lambda_k}\Big)\max_{1\leq m\leq k}m\Big(\Phi_m^{-1}\big(h(k)\big)\Big)^{p}<\infty,
$$
where $\Delta(a_k)=a_k-a_{k+1}$.
\end{theorem}
\begin{proof}
Let $f\in \Phi\text{V}[h]$, so there exists some $c>0$ such that $V_{\Phi,h}(cf)<\infty$. Without loss of generality we may assume that $c=1$. Let
$\{I_j\}_{j=1}^s$ be a nonoverlapping collection of subintervals of $[0,1]$. When $q\geq1$ we may use
Lemma \eqref{lemma1} with $x_j=|f(I_j)|$ to get
\begin{equation}
\label{Wu's estimate}
\Big(\sum_{j=1}^s |f(I_j)|^q\Big)^{\frac1{q}}\leq 16\left(1+V_{\Phi,h}(f)\right)\max_{1\leq m\leq s} m^{\frac1{q}}\Phi_m^{-1}\big(h(s)\big).
\end{equation}
In order to prove that $V_\Lambda(f)<\infty$, we need to estimate the sum $\sum_{k=1}^s \frac{|f(I_k)|^{p}}{\lambda_k}$. Taking
$x_k:=\frac1{\lambda_k}$ and $y_k:=|f(I_k)|^{p}$ in Abel's partial summation formula
$$
\sum_{k=1}^s x_ky_k=\sum_{k=1}^{s-1}\Delta(x_k)\sum_{j=1}^{k}y_j+x_s\sum_{j=1}^sy_j,
$$
one can write
$$
\sum_{k=1}^s \frac{|f(I_k)|^{p}}{\lambda_k}=\sum_{k=1}^{s-1}\Delta\Big(\frac1{\lambda_k}\Big)\sum_{j=1}^k |f(I_j)|^{p}+\frac1{\lambda_s}\sum_{j=1}^s |f(I_j)|^{p}.
$$
Then, applying (\ref{Wu's estimate}) with $q=p$ to estimate the right-hand side of the preceding equality, it follows  that
\begin{align*}
\sum_{k=1}^s \frac{|f(I_k)|^{p}}{\lambda_k}\leq\sum_{k=1}^{s-1}\Delta\Big(\frac1{\lambda_k}\Big)C^{p}\max_{1\leq m\leq k}m\Big(\Phi_m^{-1}(h(k))\Big)^{p}
+\frac1{\lambda_s}C^{p}\max_{1\leq m\leq s}m\Big(\Phi_m^{-1}(h(s))\Big)^{p}\\[.1in]
&\hspace{-13cm}\leq\sum_{k=1}^{s-1}\Delta\Big(\frac1{\lambda_k}\Big)C^{p}\max_{1\leq m\leq k}m\Big(\Phi_m^{-1}(h(k))\Big)^{p}
+\sum_{k=s}^{\infty}\Delta\Big(\frac1{\lambda_k}\Big)C^{p}\max_{1\leq m\leq k}m\Big(\Phi_m^{-1}(h(k))\Big)^{p}\\[.1in]
&\hspace{-13cm}\leq C^{p}\sum_{k=1}^\infty \Delta\Big(\frac1{\lambda_k}\Big)\max_{1\leq m\leq k}m\Big(\Phi_m^{-1}(h(k))\Big)^{p}<\infty,
\end{align*}
where $C=16\left(1+V_{\Phi,h}(f)\right)$ and the penultimate inequality is due to the fact that
$$
\frac1{\lambda_s}\max_{1\leq m\leq s}m\Big(\Phi_m^{-1}(h(s))\Big)^{p}\leq\sum_{k=s}^{\infty}\Delta\Big(\frac1{\lambda_k}\Big)\max_{1\leq m\leq k}m\Big(\Phi_m^{-1}(h(k))\Big)^{p}.
$$
This means that $f\in\Lambda\text{BV}^{(p)}$, as desired.
\end{proof}

\begin{cor}
The embedding $\Phi\emph{BV}\subseteq\Lambda\emph{BV}$ holds whenever
$$
\sum_{n=1}^\infty \Delta\Big(\frac1{\lambda_n}\Big)n\Phi_n^{-1}(1)<\infty.
$$
In particular, the embedding $\phi\Lambda\emph{BV}\subseteq\Gamma\emph{BV}$ holds whenever
$$
\sum_{n=1}^\infty \Delta\Big(\frac1{\gamma_n}\Big)n\phi^{-1}(\Lambda(n)^{-1})<\infty.
$$
\end{cor}

\begin{cor}\emph{(\cite[Theorem 2]{Avdis})}
The embedding $V[h]\subseteq \Lambda\emph{BV}$ holds whenever
$$
\sum_{n=1}^\infty \Delta\Big(\frac1{\lambda_n}\Big)h(n)<\infty.
$$
\end{cor}


\vspace{.25cm}

Recently the second author et al. \cite{GHM} obtained the following inequality and used it to characterize the embedding
$\Lambda\text{BV}^{(p)}\subseteq\Gamma\text{BV}^{(q_n\uparrow q)}$:
\begin{equation}\label{ineq}
\Big(\sum_{j=1}^n x_j^qz_j\Big)^{\frac1{q}}\leq\sum_{j=1}^n x_jy_j\max_{1\leq k\leq n}\Big(\sum_{j=1}^k z_j\Big)^{\frac1{q}}\Big(\sum_{j=1}^k y_j\Big)^{-1},
\end{equation}
where $1\leq q< \infty$, and $\{x_j\}$, $\{y_j\}$ and $\{z_j\}$ are positive nonincreasing sequences. In the sequel, we will
further exploit (\ref{ineq}) to prove the forthcoming results.

\begin{theorem}\label{2}
Let $1\leq p\leq q<\infty$. Let either $\Big\{\Gamma(n)^{\frac1{q}}/\Lambda(n)^{\frac1{p}}\Big\}$ or
$\Big\{h_2(n)^{\frac1{q}}/h_1(n)^{\frac1{p}}\Big\}$ be nondecreasing. Then the embedding $\Lambda V[h_1]^{(p)}\subseteq\Gamma V[h_2]^{(q)}$ holds whenever
\begin{equation}\label{c2}
\sup_{1\leq n<\infty}\left(\frac{\Gamma(n)}{h_2(n)}\right)^{\frac1{q}}\left(\frac{h_1(n)}{\Lambda(n)}\right)^{\frac1{p}}<\infty.
\end{equation}
\end{theorem}
\begin{proof}
Let $f\in\Lambda V[h_1]^{(p)}$ and consider a fixed $n$. Let $\{I_j\}_{j=1}^n$ be a nonoverlapping collection of subintervals of $[0,1]$. Set $x_j:=|f(I_j)|^p$,
$y_j:=1/\lambda_j$ and $z_j:=1/\gamma_j$. In view of the
equimonotonic sequences inequality \cite[Theorem 368]{Hardy} we can, and do, assume that the $x_j$ are arranged in descending order. Now, applying (\ref{ineq})
with $q/p\geq1$ in place of $q$ we obtain
$$
\Big(\sum_{j=1}^n \frac{|f(I_j)|^{q}}{\gamma_j}\Big)^{\frac{p}{q}}
\leq\sum_{j=1}^n \frac{|f(I_j)|^p}{\lambda_j}\max_{1\leq k\leq n} \frac{\Gamma(k)^{\frac{p}{q}}}{\Lambda(k)}.
$$
Therefore, we get
\begin{align*}\Big(\sum_{j=1}^n \frac{|f(I_j)|^{q}}{\gamma_j}\Big)^{\frac{1}{q}}
&\leq\Big(\sum_{j=1}^n \frac{|f(I_j)|^p}{\lambda_j}\Big)^{\frac{1}{p}}\max_{1\leq k\leq n} \frac{\Gamma(k)^{\frac{1}{q}}}{\Lambda(k)^{\frac{1}{p}}}\\
&\leq\Big(v(n;\Lambda,p,f)\Big)^{\frac{1}{p}}\max_{1\leq k\leq n} \frac{\Gamma(k)^{\frac{1}{q}}}{\Lambda(k)^{\frac{1}{p}}}\\
&\leq C.h_1(n)^{\frac{1}{p}}.\frac{h_2(n)^{\frac{1}{q}}}{h_1(n)^{\frac{1}{p}}}=C.h_2(n)^{\frac{1}{q}}.
\end{align*}
for some positive constant $C$, depending solely on $f$.
As a result, taking supremum over all collections $\{I_j\}_{j=1}^n$ as above, it follows that
$$
v(n;\Gamma,q,f)\leq C^q.h_2(n),
$$
which means that $f\in \Gamma V[h_2]^{(q)}$.
\end{proof}

An important consequence of the preceding theorem is the following result which provides a sufficient condition for the embedding $\Lambda BV\subseteq V[h]$
(see Remark \eqref{rem}).
\begin{cor}\label{WatermanChanturiya}
The embedding $\Lambda BV\subseteq V[h]$ holds whenever
\begin{equation}\label{suf}
\sup_{1\leq n<\infty}\frac{n}{\Lambda(n)h(n)}<\infty.
\end{equation}
\end{cor}
\begin{proof}
Note that $\big\{n{\Lambda(n)}^{-1}\big\}$ is nondecreasing and apply Theorem (\ref{2}) with $p=q=1$, $h_2=h$, $h_1(n)=1$ for all $n$, and $\gamma_j=1$ for all $j$.
\end{proof}

\begin{remark}\label{rem}
It is worth noting that the existence of a conditon that characterizes when $\Lambda BV$ can be embedded into $V[h]$ seems to have been unknown for a long time.
We conjecture that \eqref{suf} is a necessary condition as well.
\end{remark}
As an application of Corollary (\ref{WatermanChanturiya}), we deduce the following result by taking $h(n)=\frac{n}{\Lambda(n)}$.
\begin{cor}\emph{(\cite[Theorem 1]{Avdis})}
The following embedding holds:
$$
\Lambda BV\subseteq V[n\Lambda(n)^{-1}].
$$
\end{cor}

\begin{cor}
Let $1\leq p\leq q<\infty$. Then the embedding $\Lambda BV^{(p)}\subseteq \Gamma BV^{(q)}$ holds whenever
$$
\sup_{1\leq n<\infty}\frac{\Gamma(n)^{\frac1{q}}}{\Lambda(n)^{\frac1{p}}}<\infty.
$$
\end{cor}

\begin{cor}
The embedding $V[h_1]\subseteq V[h_2]$ holds whenever
$$
\sup_{1\leq n<\infty}\frac{h_1(n)}{h_2(n)}<\infty.
$$
\end{cor}

Let $\{p_n\}_{n=1}^\infty$ be a sequence of positive real numbers such that $1\leq p_n\uparrow p\leq\infty$.
A real-valued function $f$ on $[a,b]$ is said to be of $p_n$-$\Lambda$-bounded variation if
$$
V_\Lambda(f)=V_\Lambda(f;p_n\uparrow p):=\sup_{n\geq 1}\sup_{\{I_j\}}\Big(\sum_{j=1}^{s} \frac{|f(I_j)|^{p_n}}{\lambda_j}\Big)^{\tfrac{1}{p_n}}<\infty,
$$
where the $\{I_j\}_{j=1}^s$ are collections of nonoverlapping subintervals of $[a,b]$ such that  $\inf_{j} |I_j|\geq \frac{b-a}{2^n}$. The class of functions
of $p_n$-$\Lambda$-bounded variation is denoted by $\Lambda\text{BV}^{(p_n\uparrow p)}$. This class was introduced by Vyas in \cite{Vyas}. When $\lambda_j=1$ for all
$j$, we obtain the class $\text{BV}^{(p_n\uparrow p)}$---introduced by Kita and Yoneda \cite{KY}---which is a generalization of the well-known Wiener class
$\text{BV}_p$.

The mutual relationship between the generalized Wiener classes $\Lambda BV^{(p_n\uparrow p)}$ is rather chaotic even
in the special case where $p_n=q_n=1$ for all $n$; see \cite{Frank} for a nice and detailed discussion on this. Besides, in order to determine when
$BV^{(p_n\uparrow p)}\subseteq BV^{(q_n\uparrow q)}$ and $\Lambda BV^{(p)}\subseteq \Gamma BV^{(q_n\uparrow q)}$
(\cite[Theorem 3.1]{KY} and \cite[Theorem 1.4]{GHM}), fairly significant restrictions have been
imposed. So, it would be highly desirable to find a condition that implies the embedding
$\Lambda BV^{(p_n\uparrow p)}\subseteq \Gamma BV^{(q_n\uparrow q)}$ without any additional restrictions on the $p_n$, $q_n$, $\Lambda$ and $\Gamma$.
Theorem (\ref{th3}) provides such a condition.

Next lemma supplements \eqref{ineq} and is used in the proof of Theorem (\ref{th3}).
\begin{lemma}
If $0<q<1$, then \eqref{ineq} holds whenever the sequence $\Big\{\sum_{i=1}^k z_i/\sum_{i=1}^k y_i\Big\}_k$
is nondecreasing.
\end{lemma}
\begin{proof} First, we apply (\ref{ineq}) with $q=1$ to obtain
\begin{equation}\label{q=1}
\sum_{j=1}^n x_jz_j\leq \sum_{j=1}^n x_jy_j\max_{1\leq k\leq n} \Big(\sum_{i=1}^k z_i\Big)\Big(\sum_{i=1}^k y_i\Big)^{-1}.
\end{equation}
Then an application of the H\"older inequality yields
\begin{align*}
\sum_{j=1}^n x_j^qz_j=\sum_{j=1}^n (x_jz_j)^qz_j^{1-q}\leq\Big(\sum_{j=1}^n x_jz_j\Big)^q\Big(\sum_{j=1}^n z_j\Big)^{1-q}\\[.1in]
&\hspace{-4.6cm}\leq\Big(\sum_{j=1}^n x_jy_j\Big)^q\Big(\sum_{j=1}^n z_j\Big)^{1-q}\max_{1\leq k\leq n} \Big(\sum_{i=1}^k z_i\Big)^q\Big(\sum_{i=1}^k y_i\Big)^{-q}\\[.1in]
&\hspace{-4.6cm}\leq\Big(\sum_{j=1}^n x_jy_j\Big)^q\max_{1\leq k\leq n} \Big(\sum_{i=1}^k z_i\Big)\Big(\sum_{i=1}^k y_i\Big)^{-q},
\end{align*}
where the last two inequalities are due, respectively, to \eqref{q=1} and the fact that $\Big\{\sum_{i=1}^k z_i/\sum_{i=1}^k y_i\Big\}_k$ is nondecreasing.
\end{proof}

\begin{theorem}\label{th3}
The embedding $\Lambda BV^{(p_n\uparrow p)}\subseteq \Gamma BV^{(q_n\uparrow q)}$ holds whenever
$$
\sup_{1\leq n<\infty}\sum_{k=1}^\infty \Delta\Big(\frac1{\gamma_k}\Big)\max_{1\leq m\leq k}m\Lambda(m)^{-\frac{q_n}{p_n}}<\infty.
$$
\end{theorem}
\begin{proof}
Assume that $f\in\Lambda\text{BV}^{(p_n\uparrow p)}$. For an arbitrary but fixed $n$, let $\{I_j\}_{j=1}^s$ be a nonoverlapping collection of subintervals of $[0,1]$ with $\inf|I_j|\geq \frac1{2^n}$,
and put $q=q_n/p_n$, $x_j=|f(I_j)|^{p_n}$, $y_j=1/\lambda_j$, $z_j=1/\gamma_j$. Without loss of generality, we may also assume that the $x_j$ are arranged in
descending order. Now, by Abel's transformation and applying (\ref{ineq}) we obtain
\begin{align*}
\sum_{k=1}^s \frac{|f(I_k)|^{q_n}}{\gamma_k}=\sum_{k=1}^{s-1}\Delta\Big(\frac1{\gamma_k}\Big)\sum_{j=1}^k |f(I_j)|^{q_n}+\frac1{\gamma_s}\sum_{j=1}^s |f(I_j)|^{q_n}\\[.1in]
&\hspace{-9.5cm}\leq\sum_{k=1}^{s-1}\Delta\Big(\frac1{\gamma_k}\Big)\Big(\sum_{j=1}^k \frac{|f(I_j)|^{p_n}}{\lambda_j}\Big)^{\frac{q_n}{p_n}}
\max_{1\leq m\leq k}m\Lambda(m)^{-\frac{q_n}{p_n}}+\frac1{\gamma_s}\Big(\sum_{j=1}^s \frac{|f(I_j)|^{p_n}}{\lambda_j}\Big)^{\frac{q_n}{p_n}}
\max_{1\leq m\leq s}m\Lambda(m)^{-\frac{q_n}{p_n}}\\[.1in]
&\hspace{-9.5cm}\leq\sum_{k=1}^{s-1}\Delta\Big(\frac1{\gamma_k}\Big)V_\Lambda(f)^{q_n}
\max_{1\leq m\leq k}m\Lambda(m)^{-\frac{q_n}{p_n}}
+\frac1{\gamma_s}V_\Lambda(f)^{q_n}\max_{1\leq m\leq s}m\Lambda(m)^{-\frac{q_n}{p_n}}\\[.1in]
&\hspace{-9.5cm}\leq\sum_{k=1}^{s-1}\Delta\Big(\frac1{\gamma_k}\Big)V_\Lambda(f)^{q_n}
\max_{1\leq m\leq k}m\Lambda(m)^{-\frac{q_n}{p_n}}
+\sum_{k=s}^{\infty}\Delta\Big(\frac1{\gamma_k}\Big)V_\Lambda(f)^{q_n}
\max_{1\leq m\leq k}m\Lambda(m)^{-\frac{q_n}{p_n}}\\[.1in]
&\hspace{-9.5cm}=V_\Lambda(f)^{q_n}\sum_{k=1}^{\infty}\Delta\Big(\frac1{\gamma_k}\Big)\max_{1\leq m\leq k}m\Lambda(m)^{-\frac{q_n}{p_n}}<\infty,
\end{align*}
where we have used the fact that
$$
\frac1{\gamma_s}\max_{1\leq m\leq s}m\Lambda(m)^{-\frac{q_n}{p_n}}\leq\sum_{k=s}^{\infty}\Delta\Big(\frac1{\gamma_k}\Big)
\max_{1\leq m\leq k}m\Lambda(m)^{-\frac{q_n}{p_n}}.
$$
Taking suprema over all collections $\{I_j\}_{j=1}^s$ as above, and over all $n$ yields $V_\Gamma(f)<\infty$. That is, $f\in\Gamma BV^{(q_n\uparrow q)}$.
\end{proof}

\begin{cor}
The embedding $\Lambda BV^{(p)}\subseteq\Gamma BV^{(q)}$ holds whenever
$$
\sum_{n=1}^\infty \Delta\Big(\frac1{\gamma_n}\Big)\max_{1\leq k\leq n}\frac{k}{\Lambda(k)^{\frac{q}{p}}}<\infty.
$$
In particular, the embedding $\Lambda BV\subseteq\Gamma BV$ holds whenever
$$
\sum_{n=1}^\infty \Delta\Big(\frac1{\gamma_n}\Big)\frac{n}{\Lambda(n)}<\infty.
$$
\end{cor}

\end{document}